\documentclass[reqno,12pt,letterpaper]{amsart}
\usepackage[proof]{newmzmacros2}
\usepackage[all,cmtip]{xy}

\newcommand{\RR}{{\mathbb R}}
\newcommand{\NN}{{\mathbb N}}
\newcommand{\CC}{{\mathbb C}}

\newcommand{\CI}{{C^\infty}}

\title{Outgoing solutions via Gevrey-2 properties}

\author{Jeffrey Galkowski}
\email{j.galkowski@ucl.ac.uk}
\address{Department of Mathematics, University College London, WC1H 0AY, UK}

\author{Maciej Zworski} 
\email{zworski@math.berkeley.edu}
\address{Department of Mathematics, University of California, Berkeley, CA 94720}
\begin{document}

\begin{abstract}
Gajic--Warnick \cite{gaw} have recently proposed a definition of scattering resonances based on Gevrey-2 regularity at infinity and introduced a new class of potentials for which resonances can be defined.
 We show that 
standard methods based on complex scaling apply to a slightly larger class of potentials and show existence of resonances in larger angles.
\end{abstract}

\maketitle

\section{Introduction}

One of the main issues in theoretical and numerical scattering theory is
distinguishing outgoing parts of solutions modeling scattered waves. 
The simplest example is given by 
\begin{equation}
\label{eq:defPV}
P_V:=D_x^2+V(x) , \ \  x \in [ 1 , \infty ) , \ \ \ D_x := -i \partial_x  ,
\end{equation}
where $ V $ is a {\em compactly} supported potential and we impose {\em 
Dirichlet} boundary condition at $ x = 1 $. Clearly, 
\begin{equation}
\label{eq:PVf}   ( P_V - \lambda^2 ) u = f \in \CIc ( [ 0, \infty )) , \ 
u ( 1 ) = 0  \ \Longrightarrow \ u ( x ) = 
a e^{ i \lambda x } + b e^{ - i \lambda x } , \ \ x \gg 1 .\end{equation}
The first term is called {\em outgoing} and the second term {\em incoming}
(or vice versa, depending on your convention) -- see \cite[\S 2.1]{res}.
The method of {\em complex scaling} distinguishes incoming and outgoing solutions  by 
deforming $ x $ into a curve in $ \CC $: the $ e^{ i \lambda x } $ term becomes
exponentially decaying and the $ e^{ - i \lambda x} $ term, 
exponentially growing and hence outgoing solutions are characterized as $ L^2 $ solutions -- see \cite[\S 2.7]{res} for a simple introduction. This method, introduced in the 70's by 
Aguilar--Combes, Balslev--Combes and Simon and then widely used in 
computational chemistry,  reappeared in the 90's as
the method of {\em perfectly matched layers} in numerical analysis -- see
\cite[\S 4.7]{res} for pointers to the literature. It applies to non-compactly supported potentials as long as analyticity and decay conditions on $ V ( x ) $ 
in $ |\Im x | \leq C | \Re x | $ are imposed for $ |x| \gg 1 $. 
In the simplest setting of \eqref{eq:PVf} {\em scattering resonances} are 
$ \lambda \in \CC $ for which there exists a solution with $ f = 0 $ and $ b = 0 $. 

\renewcommand\thefootnote{\dag}%   

Gajic and Warnick \cite{gaw} have recently proposed an intriguing 
alternative for distinguishing incoming and outgoing solutions using 
Gevrey-2 regularity (see \eqref{eq:gev}) at infinity. In the setting of compactly supported
potentials their approach can be described as follows: put
$ Q_V ( \lambda ) := x^{2} e^{ - i \lambda x }( P_V - \lambda^2 ) e^{ i \lambda x} =  x^{2} ( D_x^2 + 2 \lambda D_x + V )$ and change variables to $ y = 1/x $. Then 
\eqref{eq:PVf} becomes
\begin{equation}
\label{eq:QV}
\begin{gathered}  Q_V ( \lambda ) u = ( D_y y^2 D_y - 2 \lambda D_y + 
y^{-2} V(1/y) ) u = f \in \CIc ( ( 0 , 1 ])  , \ \ u (1 ) = 0 , \\ u (y) = a + b e^{ - 2 i \lambda/y } , \ \ 0 < y \ll 1 . \end{gathered} 
\end{equation}
For $ \Im \lambda \geq 0 $, the outgoing solutions are simply
solutions which are smooth up to $ y = 0 $. For $ \Im \lambda < 0 $,
$ y \mapsto e^{ -2 i \lambda /y } \in \CI ( [0, 1]) $, however, it is
{\em not} in the Gevrey-2 class $ G^{2, |\lambda|-|\Re \lambda| + \epsilon} ( [ 0, 
1 ] ) $,
for any $ \epsilon > 0 $ (see \eqref{eq:gev} and the appendix)
that $ u \in G^{2, \sigma } ( [0,1])  $, $ \sigma \gg 1 $, thus
guarantees that $ u $ is outgoing. In particular, for $ f \equiv 0$, 
it gives a criterion for $ \lambda $ being a scattering resonance. Through a delicate Gevrey class analysis this idea was used in \cite{gaw} to 
extend definition of resonances to potentials satisfying 
$ y^{-2} V ( 1/y) \in G^{2, \sigma } ( [ 0 , 1 ] ) $, $ \sigma \gg 1 $. 
The region in which resonances are defined is shown in Fig.~\ref{fig}.

\begin{figure}

\includegraphics[width=7.5cm]{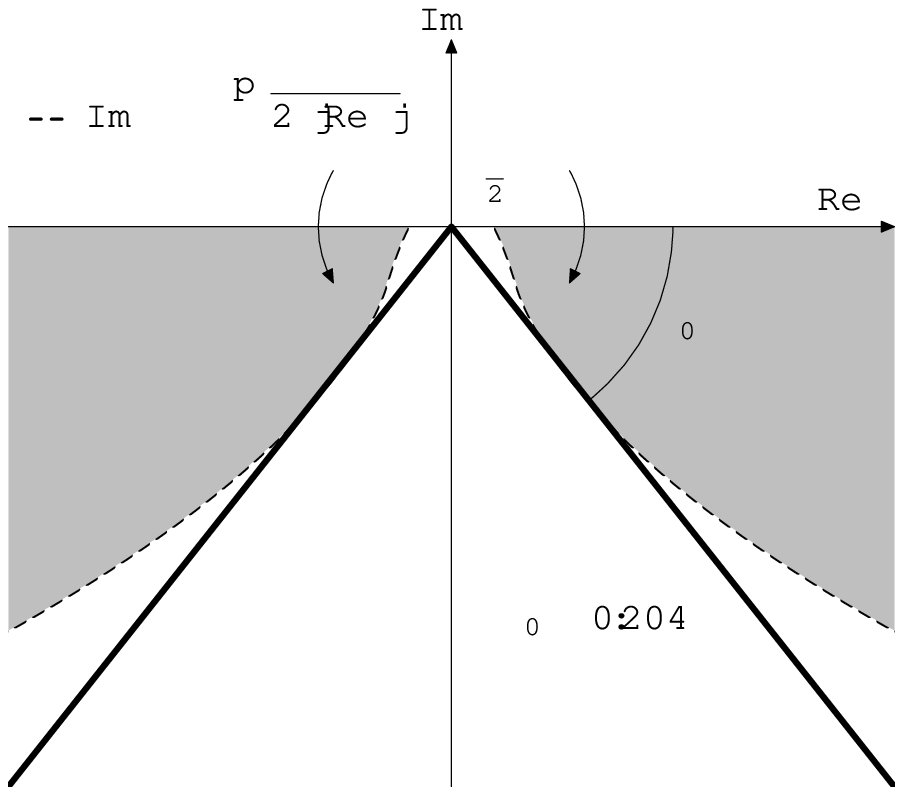}   \includegraphics[width=7.5cm]{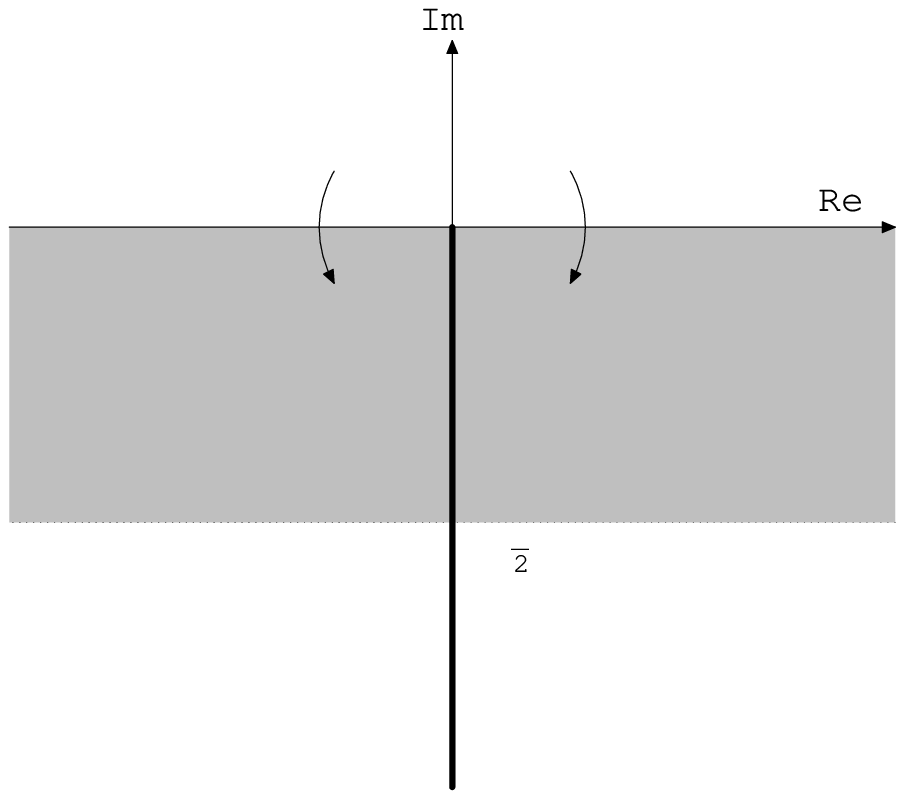}   
\label{fig}
\caption{We fix a $\sigma>0$ and on the left show the region, $\Omega_\sigma$, in which resonances are defined in 
\cite{gaw} (note the difference of conventions: here we followed
the standard convention of scattering theory \cite{res}) and on the right the illustration of the region, $\mathcal{M}_\sigma$, in this paper.
The potential satisfies
$ V ( x ) = x^{-\gamma} W ( 1/x ) $, $ W \in G^{ 2, \sigma } ( [ 0 , 1 ] ) $, with $ \gamma = 2 $ in the case of \cite{gaw} and $ \gamma =1 $
in this paper. For a fixed $ \sigma $ the region on the left is larger 
for $ |\Re \lambda | \gg 1 $ but as $ \sigma \to \infty $ the conic regions are dramatically different. The asymptotic 
boundary of $ \Omega_\sigma $ is heuristically explained using \eqref{eq:QV} and the fact that $ e^{-2 i \lambda/ x} \in \widehat G^{2, |\lambda | - |\Re \lambda | } $ (see the appendix) so that for a fixed $ \sigma$, we obtain
the curve $ | \lambda | - | \Re \lambda | = \sigma $, that 
is $ 2 \sigma | \Re \lambda | =  |\Im \lambda|^2 - \sigma^2 $.
 }
\end{figure}

In this note we observe that the potentials of the form considered in 
\cite{gaw} (in fact, potentials with slower decay at infity) can be
decomposed into a sum of a potential analytic in a large sector and an exponentially decaying potential. We then use complex scaling method
and consider the exponentially decaying potential as a perturbation. 
Using essentially well-know results this easily gives meromorphic continuation to 
strips in $ \CC \setminus i ( - \infty, 0 ] ) $.

To formulate the result we define Gevrey classes, in particular the Beurling--Gevrey class,
$ \gamma^a ( [ 0 , 1 ] )$: 
\begin{equation}
\label{eq:gev}
\begin{gathered}
G^{ a , \sigma } ( [ 0, 1 ] ) := \{ 
u \in \CI ( [ 0, 1 ] ) \, \mid \,  \exists \, C , \ \sup_{ x \in [ 0 , 1 ] }
| D_x^n u ( x ) | \leq C \sigma^{-n} (n!)^a  \}, \\
G^a := \bigcup_{ \sigma > 0 } G^{a, \sigma } , \ \ 
\gamma^a := \bigcap_{ \sigma > 0 } G^{ a, \sigma } , \ \ 
\widehat G^{a,\sigma} := G^{a,\sigma } \setminus \bigcup_{ \sigma' > \sigma}
G^{a, \sigma' }, \ \ 
a \geq 1 ,  \ \sigma > 0 ,
\end{gathered}
\end{equation}
see \cite[\S 1.3]{H1} for some fundamental results and references. With this in place we can state

\medskip

\noindent
%\begin{theo} 
%\label{t:main}
{\bf Theorem.} {\em Suppose that $ V ( x ) = 
x^{-1} W ( 1/x ) $ with
$ W \in G^{2,\sigma} ( [ 0 , 1 ] ) $. Then the scattering resolvent, 
$$
R_V(\lambda):=(P_V-\lambda^2)^{-1}:L^2_{\comp} ( [ 1, \infty ) \to L^2_{\loc} ( [ 1, \infty ) ),  \ \ \Im \lambda \gg 1 , 
$$
continues to a meromorphic family of operators with  poles of finite rank on 
$$
 \mathcal{M}_\sigma:=\{ \lambda \in \CC \, \mid \, \Im \lambda > - \tfrac \sigma 2 \}\setminus i ( - \infty , 0] .
$$
In particular, when $ W \in \gamma^2 ( [ 0 , 1 ] ) $, the resolvent 
continues to $ \CC \setminus i ( - \infty , 0] $.
}

\medskip

\noindent
{\bf Remarks.} 1. The motivation in \cite{gaw} came from the study of
quasinormal modes of black holes \cite{gaw1}. In the case considered there, unlike in de Sitter 
space situations (positive cosmological constant -- see Bony--H\"afner \cite{boh}, Dyatlov \cite{dyb} and Vasy \cite{vas} 
for detailed studies and references), the issue of seeing the effects of 
quasinormal modes on wave evolution is complicated by the behaviour at 
$ 0 $. However, as shown by Dyatlov \cite{dyc}, some results can be obtained by considering frequency cut-offs. For fine decay results related to the behaviour of the resolvent at zero see Hintz \cite{hin} and references given there.

\noindent
2. The necessity of a cut at $ 0 $ or a continuation to (a part of) the Riemann surface of $ \lambda \mapsto \log \lambda $ is a genuine 
phenomenon as can be seen by considering $ V ( x ) = \alpha/ x^2 $.
For instance, for $ \alpha = n^2 - \frac14 $, $ P_V $ corresponds to the 
radial component of the Dirichlet Laplacian on $ \RR^2 \setminus 
B( 0 , 1 ) $ acting on the $ n$th Fourier component -- see Christiansen
\cite{christ} for recent advances in scattering by obstacles in even dimensions and references. The proof here would allow moving to 
an angle $ \pi $ on the Riemann surface (rather than $ \pi/2 $ as stated) 
at the expense of shrinking the strip (which gives the whole Riemann surface $ - \pi < \arg \lambda < 2 \pi $.

\noindent
3. Defining resonances  for potentials which are not-analytic in conic neighbourhoods of infinity, or rather seeing their impact on ``observable" phenomena, 
has been an object of recent studies. Martinez--Ramond--Sj\"ostrand \cite{mars} 
proved that in in some non-analytic situations resonances are invariantly defined up to any power of their imaginary part. Bony--Michel--Ramond \cite{bon} showed that
many phenomena associated to resonances can be seen for potentials decaying at infinity: the ``interaction region" (where, say, the trapped set 
lies in the case of AdS black holes) is responsible for ``observable" effects.

\noindent
4. We expect that the improved angle estimate of this note and 
the better large $ \Re \lambda $ region can be combined. One approach would require the methods of Helffer--Sj\"ostrand \cite{HS} recently revisited
by Guedes Bonthonneau--J\'ez\'equel \cite{Guj} and the authors 
\cite{gaz1}. However, without stronger motivation, we have restricted ourselves to more straightforward methods.

The paper is organized as follows. In \S \ref{gev} we revisit 
well known asymptotic results of Borel, Ritt and Watson to decompose 
potentials in the theorem. In \S \ref{exp} we show that having 
a meromorphic extension guarantees exponential growth estimates for 
the resolvent. That meromorphic continuation is obtained using the method
of complex scaling for the analytic part of the potential. Finally, in 
\S \ref{mer}, we use this exponential estimates to show that adding an
exponentially decay potential given meromorphy in a strip in $ \CC 
\setminus i ( -\infty, 0 ] $. Except for invoking the now standard method of complex scaling, the paper is essentially self contained. 

\medskip\noindent\textbf{Acknowledgements.} Partial support for M.Z. by the National Science Foundation grant DMS-1500852 is also  
gratefully acknowledged. 

\section{Gevrey-2 property at infinity}
\label{gev} 

In this section, we decompose a $G^{2,\sigma}$ potential into one which is analytic in a large sector and one which decays exponentially towards infinity.

\begin{prop}
\label{p:decomp}
For $ V \in \CI ( [ 1, \infty ) ) $ satisfying
%\begin{equation}
%\label{eq:W2V}  
$ V ( x ) = W ( 1/x ) $, $ W \in G^{2, \sigma} ( [ 0 , 1 ] ) $ %\end{equation}
and any $ \rho < \sigma $ and $ \epsilon > 0 $ we have
\begin{equation}
\label{eq:decomp}
\begin{gathered}
 V ( x ) = V_1 ( x ) + V_2 ( x ) , \ \ 
 V_1 \in \CI ( [ 0 , \infty )),  \ \ 
 %\forall \, \ell \in \NN  \, 
 \exists \, C  \ \ | V_1%^{(\ell)} 
 ( x ) | 
 \leq C e^{ - \rho x  } , \\
\text{ $ V_2 ( z )$  is holomorphic in $ \mathcal C_{\epsilon} $
and } \\ V_2 ( z) \sim \sum_{\ell = 0 }^\infty \frac{ W^{(\ell)} ( 0 ) }
{ \ell!} {z^{-\ell}} , \ z \to \infty  , \ z \in \mathcal C_{\epsilon} , 
\end{gathered}
\end{equation}
where $ \mathcal C_{\epsilon} := \{ z : |\arg z| < \pi/2 -\epsilon \} $. 
\end{prop}
\begin{proof}
We observe that the decomposition \eqref{eq:decomp} follows from 
the following decomposition of $ W $: for all $ \rho < \sigma$, 
\begin{equation}
\label{eq:decomp2} 
\begin{gathered}  W ( x ) = W_1 ( x ) + W_2 ( x ) , \ \ 
W_1 \in G^{ 2 , \rho } ( [ 0, 1 ]) , \ \ \rho < \sigma \ \ \forall \, \ell \in \NN \ \ W_1^{(\ell)} ( 0 ) = 0 , \\
\text{ $ W_2 ( z )$  is holomorphic in $ \mathcal C_\epsilon $ 
and } \ W_2 ( z) \sim \sum_{\ell = 0 }^\infty \frac{ W^{(\ell)} ( 0 ) }
{ \ell!} {z^{\ell}} , \ z \to 0  , \ z \in \mathcal C_{\epsilon} . 
\end{gathered}
\end{equation}
In fact, if $ V_2 ( z ) := W_2 ( 1/z ) $ then the expansion of $ V_2 $ follows. On the other hand, Taylor expansion at $ 0 $ and 
the Gevrey property imply that
\[  | W_1 ( x ) | \leq \min_{ \ell \in \NN} x^\ell \sup_{ [ 0 , x ] } | W_1^{(\ell)} |/ \ell! \leq C \min_{\ell \in \NN} (x/\rho)^\ell \ell! \leq C' e^{ - \rho /x  } , \]
which gives the required estimate for $ V_1 ( x ) = W_1 ( 1/x ) $.

Hence we need to establish \eqref{eq:decomp2} and for that we use the classical approach of Watson \cite{wat} (see \cite[Chapter]{bals} for a modern presentation). To implement it, we put
\begin{equation}
\label{eq:fn}  f_n := \frac{ W^{(n)} ( 0 ) } {n!}\,, \ \ \   |f_n | \leq K \sigma^{-n} 
n! \,, \ \ \   g ( \zeta) := \sum_{ n=0}^\infty \frac{ f_n \zeta^n}{n!} .
\end{equation}
The function $ g $ is then holomorphic in $ |\zeta | < \sigma $ and for
$ 0 < \rho_1< \sigma $ we define 
\begin{equation}
\label{eq:defW2}
  W_2 ( z ) := z^{-1} \int_0^{\rho_1}  g ( \zeta ) e^{-\zeta/z } d \zeta , 
\end{equation}
which is holomorphic for $ z \in \Lambda $, the Riemann surface of $ 
\log z $. Then for $ \cos ( \arg z ) \geq \epsilon   $, 
\[  | W_2 ( z ) - \sum_{ n=0}^{N-1} f_n z^n  | %\leq |z|^N \epsilon^{-1 - N}
%N!  K \sum_{ n=0}^N  \sigma^{-n} \rho^{n-N} 
\leq  C_{\rho_1}K N \epsilon^{-1}  |z|^N
 N!  ( \epsilon \rho_1 )^{-N}  .\]
 
 Indeed, 
 \begin{align*}
&\Big|W_2(z)-\sum_{n=0}^{N-1}f_nz^n\Big|
\leq \Big|z^{-1}\int_0^{\rho_1} g(\zeta )e^{-\zeta/z}d\zeta-\sum_{n=0}^{N-1}\int_0^\infty z^{-1}\frac{f_n}{n!}\zeta^ne^{-\zeta/z}d\zeta\Big| \\
&\ \ \ \  \leq \frac{\sigma}{\sigma-\rho_1}\int_0^{\rho_1} |z|^{-1}K\zeta^N \sigma^{-N} e^{-\zeta\epsilon/|z|}d\zeta +\int_{\rho_1}^{\infty} |z|^{-1}K\sum_{n=0}^{N-1}\sigma^{-n}\zeta^ne^{-\zeta\epsilon/|z|}d\zeta\\
&\ \ \ \ \leq \frac{\sigma}{\sigma-\rho_1}K\epsilon^{-1}|z|^NN!(\epsilon\sigma)^{-N}+\int_{\rho_1}^\infty KN|z|^{-1}(1+(\sigma^{-1}\zeta)^{N})e^{-\zeta \epsilon/|z|}d\zeta\\
&\ \ \ \ \leq \frac{\sigma}{\sigma-\rho_1}K\epsilon^{-1}|z|^NN!(\epsilon\sigma)^{-N}+K\epsilon^{-1}Ne^{-\epsilon\rho_1/|z|}+KN\epsilon^{-1}|z|^NN!(\epsilon\sigma)^{-N}\\
&\ \ \ \ \leq C_{\rho_1}\epsilon^{-1}|z|^NN!(\epsilon\rho_1)^{-N}
 \end{align*}
(See \cite[Exercise 3, p.16]{bals}.) This shows the existence of the expansion in \eqref{eq:decomp}.

Since this expansion holds with all derivatives, on the real axis (we can then take $ \epsilon $ arbitrarily close to $ 1 $) we have
$ |W_2^{(\ell)} ( x ) | \leq C_\rho (\ell!)^2 \rho^{-\ell} $ for any $
\rho < \rho_1  $. Since the expansion also shows that 
$ W_2^{(\ell)} ( 0 ) = W^{(\ell) } ( 0 )$, \eqref{eq:decomp2}, and hence
\eqref{eq:decomp}, follow.
\end{proof}

\section{Meromorphic continuation}

Our main theorem is a consequence of the following proposition together with Proposition~\ref{p:decomp}.
\begin{prop}
\label{p:mer}
Suppose that $ V \in L^\infty ( [1,\infty) ; \mathbb{C} ) $ can be decomposed
as follows 
\begin{equation}
\label{eq:decomp3} 
\begin{gathered} V ( x ) = V_1 ( x ) + V_2 ( x ) , \ \ 
| V_1 ( x ) | \leq C e^{ - \gamma x } , \\
\text{ $ V_2 ( z )$  is holomorphic for $ | \arg z | < \alpha < \pi $ and 
$ V_2 ( z ) \underset{z\to \infty}{\longrightarrow} 0 $ there.}\\
%\text{ $ V_2 ( -z )$  is holomorphic for $ | \arg z | < \alpha < \pi $ and 
%$ V_2 ( -z )\underset{z\to \infty}{\longrightarrow} 0 $ there.}
\end{gathered}
\end{equation}
Then $ R ( \lambda ) $  has a meromorphic continuation to 
$  \{ \lambda :  -\alpha <\arg \lambda <\pi+\alpha , \ 2\Im \lambda > 
- \gamma  \} $. 
\end{prop}

Throughout this section, for $V\in L^\infty(\mathbb{R};\mathbb{C})$, we let
$
P_V:=D_x^2+V
$
be defined by the quadratic form 
$$
Q(u,v)=\langle D_xu,D_x v\rangle_{L^2(1,\infty)}+\langle Vu,v\rangle_{L^2(1,\infty)}
$$
with form domain $H_0^1(1,\infty)$. Then, for $\Im\lambda\gg 1$, 
$$
R_{_{V}}(\lambda):=(P_{V}-\lambda^2)^{-1}:L^2\to L^2
$$
the \emph{resolvent of $P_{V}$} is defined by spectral theory.

In order to prove Proposition~\ref{p:mer} we first observe that the now standard method of complex scaling (see Sj\"ostrand \cite{Sj96} for 
a presentation from the PDE point of view and references)
shows that 
$$
R_{V_2}(\lambda):L^2_{\comp} ( [0 , \infty ) \to L^2_{\loc} ( [ 0 , \infty) )
$$
continues from $\Im \lambda \gg 1$ to $-\alpha<\arg \lambda <\pi +\alpha$.

\subsection{Exponential estimates on $R_{_{V_2}}$}
\label{exp}

The goal of this section is to prove the following proposition. A general case would follow from Helffer--Sj\"ostrand theory \cite{HS} but we opt for a quick one dimensional argument. 
\begin{prop}
\label{p:expEst}
Fix $\gamma>0$. Then, the meromorphic 
family  $ \lambda \mapsto R_{_{V_2}} ( \lambda ) : L^2_{\comp} \to L^2_{\loc} $ extends to a meromorphic family 
\begin{equation}
\label{eq:gammer} \lambda\mapsto R_{_{V_2}}(\lambda): e^{-\gamma \langle x\rangle }L^2\to e^{\gamma \langle x\rangle}H^2,
\end{equation}
for 
$$  \{\lambda \in \mathbb{C} \mid \Im \lambda >-\gamma \} \cap \mathcal{W}_\alpha, \ \ 
\mathcal{W}_\alpha:=\{ -\alpha <\arg \lambda <\pi +\alpha\}.
$$
\end{prop}

We first need the following apriori estimates.
\begin{lemm}
\label{l:apriori}
Let $V\in L^\infty([1,\infty);\mathbb{C})$ such that $\lim_{x\to \infty}V(x)=0$ and fix $\gamma>0$ Then, for all $\lambda$ with $\gamma> |\Im \lambda|$, $\lambda\neq 0$, there is $C(\lambda) >0$ such that for all $u\in H^2_{\loc}$, 
\begin{gather}
\label{e:est1}\|e^{-\gamma\langle x\rangle }u\|_{H^2}\leq C(\lambda) \left(\|e^{\gamma\langle x\rangle }(P_V-\lambda^2)u\|_{L^2}+\|u\|_{L^2(1,4)}\right).
%\\ \label{e:est2}\|e^{-\gamma\langle x\rangle }u\|_{H^2}\leq C(\|e^{\gamma\langle x\rangle }(P^*-\bar{\lambda}^2)u\|_{L^2}+\|u\|_{L^2(-2,2)}).
\end{gather}
%In particular, $(P-\lambda^2): \mathcal{Y}\to  e^{-\gamma\langle x\rangle}L^2$ is Fredholm where 
%$$
%\mathcal{Y}=\{u\in e^{\gamma\langle x\rangle}H^2\mid (P-\lambda^2)u\in e^{-\gamma\langle x\rangle}L^2\}.
%$$
%\red{Jeff needs to check this functional analysis}

 \end{lemm}
\begin{proof}
First, suppose that $u\in C_c^\infty$ and 
$
(P_V-\lambda^2)u=f
$.
Then, with
$$
w_\lambda(x):=\begin{pmatrix}u(x)\\\lambda^{-1}D_xu(x)\end{pmatrix},
$$
we have 
$$
D_xw_\lambda = A(x)w_\lambda+\begin{pmatrix}0 \\ \lambda^{-1}f\end{pmatrix},\qquad A(x):=\begin{pmatrix} 0&\lambda\\ \lambda -\lambda^{-1}V&0\end{pmatrix}.
$$
Now, define $U_s(x):\mathbb{C}^2\to \mathbb{C}^2$ such that  
$
D_x U_s(x)=A(x)U_s (x)$, $ U_s(s)=\Id$. 
Then, 
$$
\partial_x \|U_s(x)u_0\|^2=\langle i (A-A^*)U_s(x)u_0, U_s(x)u_0 \rangle.
$$
The eigenvalues of $i(A-A^*)(x)$ are given by
$$
\pm\sigma(x)=\pm |2\Im \lambda -i\lambda^{-1}V(x)|
$$
Note that, since $\lim_{x\to \infty}V=0$, and $|\lambda|>0$, there is $R>1$ such that for $x>R$ $\sigma(x)<(\gamma+|\Im \lambda|)$. Therefore, for $x>s$,
\begin{align*}
\|U_s(x)u_0\|&\leq e^{\frac{1}{2}\int_s^x \sigma(t)dt}\|u_0\|\leq  Ce^{\frac{1}{2}(\gamma+|\Im \lambda|)(x-s)}\|u_0\|.
\end{align*}
Using a similar argument for $x<s$, we have
$$
\|U_s(x)u_0\|\leq C e^{\frac{1}{2}(\gamma+|\Im \lambda|)|x-s|}\|u_0\|\\
$$

Next, observe that 
$$
w_\lambda(x) = U_s(x)w_\lambda(s)+i\int_s^x U_t(x)\begin{pmatrix}0\\\lambda^{-1}f(t)\end{pmatrix}dt
$$
Therefore,  for all $s\in \mathbb{R}$,
\begin{align*}
e^{-\gamma|x|}\|w_\lambda(x)\|&\leq C e^{\frac{1}{2}(\gamma+|\Im \lambda|)|x-s|-a|x|}\|w_\lambda(s)\|+C\int_s^x e^{\frac{1}{2}(\gamma+|\Im \lambda|)(|x-t|)-a|x|}|f(t)|dt\\
%&\leq C e^{\frac{1}{2}(\gamma+|\Im \lambda|)|x-s|-\gamma|x|}\|w_\lambda(s)\|\\
%&\qquad+C\sqrt{\int_s^x e^{-(\gamma-|\Im \lambda|)(|x|+|t|)}dt\int e^{2\gamma|t|}|f(t)|^2dt}\\
&\leq C(e^{\frac{1}{2}(\gamma+|\Im \lambda|)|x-s|-\gamma|x|}\|w_\lambda(s)\|+e^{-\frac{1}{2}(\gamma-|\Im \lambda|)|x|}\|e^{\gamma|x|}f\|_{L^2(s,x)})
\end{align*}
Therefore averaging in $s\in (2,3) $, we have
$$
e^{-2\gamma|x|}\|w_\lambda(x)\|^2\leq C(e^{-(\gamma-|\Im \lambda|)|x|}\|w_\lambda\|_{L^2(2,3)}^2+ e^{-(\gamma-|\Im \lambda|)|x|}\|e^{\gamma|\cdot|}f(\cdot)\|^2_{L^2(1\leq \cdot \leq \max(3,|x|))}).
$$
Finally integrating in $x$, we obtain for $R>3$,
$$
\|e^{-\gamma|x|}w_\lambda\|_{L^2([1,R)}\leq C(\|e^{\gamma |x|}f\|_{L^2(1,R)}+\|w_\lambda\|_{L^2(2,3)}).
$$
In particular, for $u\in C_c^\infty([1,\infty))$, 
\begin{equation}
\label{e:step1}
\|e^{-\gamma\langle x\rangle}u\|_{H^1(1,R)}\leq C(\|e^{\gamma\langle x\rangle }(P_V-\lambda^2)u\|_{L^2([1,R))}+\|u\|_{H^1(2,3)}.
\end{equation}

Now, let $u\in H^2_{\loc}$ with $e^{\gamma\langle x\rangle }(P_V-\lambda^2)u\in L^2$ and $\chi \in C_c^\infty(\mathbb{R})$ with $\chi \equiv 1$ on $[-1,1]$. Then, let $u_{n,k} \in C_c^\infty([1,\infty)$ with 
$$
u_{n,k}\overset{H^2}{\underset{k\to \infty}{\longrightarrow}}\chi(n^{-1}x)u
$$

Fix $R>0$. Then, letting $n>R$ and applying~\eqref{e:step1} with $u=u_{n,k}$,  we obtain
$$
\|e^{-\gamma\langle x\rangle} u_{n,k}\|_{H^1}\leq C(\|e^{\gamma\langle x\rangle}(P_V-\lambda^2)u_{n,k}\|_{L^2(1,R)}+\|u_{n,k}\|_{H^1(2,3)}).
$$
Sending $k\to \infty$ and using that $\chi(n^{-1}x)\equiv 1$ on $[1,R)$, we have
$$
\|e^{-\gamma\langle x\rangle}u\|_{H^1}\leq C(\|e^{\gamma\langle x\rangle}(P_V-\lambda^2)u\|_{L^2(1,R)}+\|u\|_{H^1(2,3)}).
$$
Sending $R\to \infty$ we obtain
$$
\|e^{-\gamma\langle x\rangle}u\|_{H^1}\leq C(\|e^{\gamma\langle x\rangle}(P_V-\lambda^2)u\|_{L^2}+\|u\|_{H^1(2,3)}).
$$
Finally, to complete the proof of~\eqref{e:est1} observe that
$$
\|u\|_{H^2(2,3)}\leq C( \|(P_V-\lambda^2)u\|_{L^2(1,4)}+\|u\|_{H^1(1,4)}).
$$
and similarly
\[
\begin{split}
\|e^{-\gamma\langle x\rangle }u\|_{H^2}  & \leq C\|e^{-\gamma\langle x\rangle}u\|_{H^1}+C\|e^{-\gamma\langle x\rangle}\partial_x^2u\|_{L^2}\\
& \leq  C \|e^{-\gamma\langle x\rangle}u\|_{H^1}+C\|e^{-\gamma\langle x\rangle}(P_V-\lambda^2)u\|_{L^2}. 
\end{split} \]
%Note that~\eqref{e:est2} follows by the same argument with $V_2$ replaced by $\bar{V}_2$ and $\lambda$ by $\bar{\lambda}$.
This completes the proof of \eqref{e:est1}.
\end{proof}

%\begin{proof}[Proof of Proposition~\ref{p:expEst}]
%Since $(P-\lambda^2):L^2\to L^2$ is invertible away from a discrete subset of $\Im \lambda >0$ we have by the analytic Fredholm theorem that 
%$$
%\tilde{R}(\lambda):= (P-\lambda^2)^{-1}:e^{-\gamma\langle x\rangle}L^2\to e^{\gamma\langle x\rangle}L^2
%$$
%is a meromorphic family of operators for $|\Im \lambda |<\gamma$. Now, for $u,v\in L^2_{\comp}$ and $\Im \lambda >0$, 
%$$
%\langle \tilde{R}(\lambda)u,v\rangle_{L^2}=\langle R_2(\lambda)u,v\rangle_{L^2}.
%$$
%By the meromorphic continuation, this equality continues to $\{|\Im \lambda|<\gamma\}\cap\{\arg \lambda <\alpha\}$.
%Since $L^2_{\comp}$ is dense in $e^{-\gamma\langle x\rangle}L^2$, we have that $\tilde{R}(\lambda)u=R_2(\lambda)u$ and $R_2$ extends by density as claimed.
%\end{proof}

\begin{proof}[Proof of Proposition \ref{p:expEst}]
Applying Lemma \ref{l:apriori} with $ f \in L^2_{\rm{comp}} $
and $ u = R_{_{V_2}} ( \lambda ) f $ 
shows that for $ \Im \lambda > - \gamma $ 
the meromorphic family $ R_{_{V_2}}  ( \lambda ) : L^2_{\rm{comp}}  \to 
L^2_{\rm{loc}}  $ has a better mapping property: 
\[ R_{_{V_2}} ( \lambda ) : L^2_{\comp } \to e^{ \gamma \langle x \rangle } L^2  .\]
In particular for $ \chi \in \CIc ( \RR ) $ and 
$ g \in e^{-\gamma \langle x \rangle } L^2 ( \RR ) $, 
\begin{equation}
\label{eq:langle}
\begin{aligned} | \langle R_{_{V_2}} ( \lambda ) \chi f , g \rangle_{L^2} | & \leq( \|e^{\gamma\langle x\rangle}\chi f\|_{L^2}+\|R_{_{V_2}}(\lambda)\chi f\|_{L^2(1,4)}) \| g \|_{ e^{ - \gamma\langle x \rangle } L^2 }\\
&\leq 
C ( \lambda ) \| f \|_{ L^2 } \| g \|_{ e^{ - \gamma\langle x \rangle } L^2 } , \ \ 
\Im \lambda > - \gamma \, 
\end{aligned}
\end{equation}
where $ C ( \lambda ) $ is bounded on compact subsets of 
$ \mathcal W_\alpha \setminus {\rm{Res}} ( P_{V_2} ) $ (where $\rm{Res}$ denotes the set of resonances). 
On the other hand, for $ \Im \lambda > 0 $, 
\[ \langle R_{_{V_2}} ( - \bar \lambda ) \chi f , g \rangle_{L^2} =
\langle f , \chi R_{_{V_2}} ( - \bar \lambda )^* g \rangle_{L^2 } = 
\langle f , \chi R_{_{V_2}} ( \lambda ) g \rangle_{L^2 } . \]
But then, \eqref{eq:langle} and analytic continuation show that 
\[  \chi R_{_{V_2}} ( \lambda ) : e^{ - \gamma \langle x \rangle } L^2 ( \RR) 
\to L^2 ( \RR ) , \ \ \chi \in \CIc ( \RR ) . \]
We can then apply Lemma~\ref{l:apriori} again to see 
that \eqref{eq:gammer} holds.
\end{proof}

\subsection{Completion of the proof of Proposition~\ref{p:mer}}
\label{mer}

%\begin{proof}[Proof of Proposition \ref{p:mer}]
We first observe that with $V=V_1+V_2$,
$$
(P_V-\lambda^2)R_{_{V_2}}(\lambda)=I+V_1R_{_{V_2}}(\lambda).
$$
Then, for $\Im \lambda\gg 1$, $I+V_1R_{_{V_2}}(\lambda):L^2\to L^2$ is invertible by Neumann series and hence
$$
R_{_{V}}(\lambda)=R_{_{V_2}}(\lambda)(I+V_1R_{_{V_2}}(\lambda))^{-1}, \qquad \Im \lambda \gg 1.
$$
For $\Im \lambda>0$, 
$
R_{_{V_2}}(\lambda):L^2\to H^2
$
and hence
$
e^{-\gamma\langle x\rangle}R_{_{V_2}}(\lambda):L^2\to L^2
$
is compact. In particular, since $e^{\gamma\langle x\rangle} V_1\in L^\infty$, 
$
I+V_1R_{_{V_2}}(\lambda):L^2\to L^2
$
is a Fredholm operator of index 0 and we have by the analytic Fredholm theorem, (see e,g,~\cite[Appendix C.3]{res})
$$
R_{_V}(\lambda)=R_{_{V_2}}(\lambda)(I+V_1R_{_{V_2}}(\lambda))^{-1}:L^2\to L^2,\qquad \Im \lambda>0,
$$
is a meromorphic family of operators with finite rank poles.

Next, consider $\lambda \in \mathcal{W}_\alpha\cap\{2|\Im \lambda|<\gamma\}$. Then, letting $\gamma'>0$ such that $2|\Im \lambda|<\gamma'<\gamma$, we have by Proposition~\ref{p:expEst}
$$
e^{-\gamma \langle x\rangle}R_{_{V_2}}(\lambda):e^{-\gamma'\langle x\rangle/2}L^2\to e^{-(\gamma'+2(\gamma-\gamma'))\langle x\rangle/2}H^2
$$
In particular, since $\gamma'<\gamma$,
$$
e^{-\gamma \langle x\rangle}R_{_{V_2}}(\lambda):e^{-\gamma'\langle x\rangle/2}L^2\to e^{-\gamma'\langle x\rangle/2}L^2
$$
is compact and hence, using that $e^{\gamma \langle x\rangle} V_1\in L^\infty$, we have
$$
I+V_1R_{_{V_2}}(\lambda):e^{-\gamma'\langle x\rangle/2}L^2\to e^{-\gamma'\langle x\rangle/2}L^2
$$
is Fredholm. Therefore, by the analytic Fredholm theorem and Proposition~\ref{p:expEst}, for all $\lambda \in \mathcal{W}_\alpha$ with $2|\Im \lambda|<\gamma'$, 
$$
R_{_{V}}(\lambda)=R_{_{V_2}}(\lambda)(I+V_1R_{_{V_2}}(\lambda))^{-1}:e^{-\gamma'\langle x\rangle/2}L^2\to e^{\gamma'\langle x\rangle/2}L^2,
$$
is a meromorphic family of operators with finite rank poles.\qed
%\end{proof}

\vspace{0.4cm}
\begin{center}
\noindent
{\sc  Appendix: Gevrey properties of exponentials.}
\end{center}
%\vspace{0.1cm}
\renewcommand{\theequation}{A.\arabic{equation}}
\refstepcounter{section}
\renewcommand{\thesection}{A}
\setcounter{equation}{0}

To explain the asymptotics appearing in Fig.\ref{fig}, we show (in the notation of \eqref{eq:gev}) that for 
$ \Im z > 0 $
\begin{equation}
\label{eq:gez}  x \mapsto e^{ i z /x } \in 
\widehat G^{2, \sigma ( z ) } ( [ 0 , \infty) ) , \ \ \ \
\sigma ( z ):={\frac{|z|-|\Re z|}{2}}. 
\end{equation}
To prove this we need the following 
characterization of $ G^{2, \sigma} ( \RR ) \cap \CIc ( \RR ) $:
(see \cite[Lemma 12.7.4]{H2} for a similar argument):
for $ u \in \CIc ( \RR ) $ 
\begin{equation}
\label{eq:ges} 
\exists \, C  \ | \widehat u ( \xi ) | \leq C  e^{ - 2 |  \sigma \xi|^{\frac12} } \ \Rightarrow \ 
u \in G^{2, \sigma } ( \RR ) \ \Rightarrow \ 
\exists \, C  \ | \widehat u ( \xi ) | \leq C \langle \xi \rangle^{\frac12}  e^{ - 2 |  \sigma \xi|^{\frac12} } 
. \end{equation}
\begin{proof}[Proof of \eqref{eq:ges}]
We start by estimating $ D_x^n u $ under the assumption on $ \widehat u $:
\[  \begin{split} | D_x^n u ( x ) | & \leq \int_{\RR}  |\xi|^n | \widehat u ( \xi ) | d \xi 
\leq C (4 \sigma)^{-n} \int_{\RR} | ( 4 \sigma |\xi|)^{\frac12 } )^{2n}  e^{ - ( 4 \sigma |\xi|)^{\frac12} }  d \xi \\ & 
= 
4 C (4 \sigma)^{-n-1 } \int_{0}^\infty t^{2n+1} e^{ -t} dt %= 4 
%C (4 \sigma)^{-n-1} ( 2 n )! \\
%& 
= C (4 \sigma) ^{-n} \frac{( 2 n )!}{ (n!)^2 } (n!)^2 
% \leq C ( 4 \sigma )^{-n} \frac{ e}{ \sqrt {2 \pi}} 
%\frac{( 2 n )^{2n + \frac12}}{ n^{2n + 1}} (n!)^2 
\leq 
C' n^{-\frac12 } \sigma^{-n} (n!)^2 , 
\end{split} \]
where in the last inequality we used Stirling's approximation.

On the other hand, if $ u \in G^{2, \sigma} \cap \CIc $ then, for 
$ |\xi| > 1 $, 
\[ \begin{split}  | \widehat u ( \xi ) | & = \left| \int_\RR  D_x^n u ( x ) \xi^{-n} e^{- i 
x \xi } d x \right| \leq 
C |\xi|^{-n} \sup | D^n_x u | \leq C' |\xi| ^{-n}  \sigma^{-n} ( n!)^2 
\\
& \leq C'' n \left( \frac { n }{ | \sigma \xi|^{\frac12} } \right)^{2n} 
e^{-2n} \leq C''' |\xi|^{\frac12} e^{ - 2 | \sigma \xi|^{\frac12} } ,  \end{split} \]
where again we used Stirling's formula and to obtain the last inequality we chose $ n = [ | \sigma \xi |^{\frac12} ] $.
\end{proof}

\begin{proof}[Proof of \eqref{eq:gez}]
We start by computing the asymptotics of the (distributional) Fourier transform of $ w ( x ) := x_+^0 e^{ i z /x } $ for a {\em fixed} $ z$ 
with $ 0 < \arg z <  \pi $.
For $ \xi  \geq 1 $ we deform the contour to $x=-e^{i\varphi_+}y$ with $\varphi_+:=\frac{\arg (z)+\pi}{2}$ to obtain 
\[  \begin{split} 
\widehat w ( \xi ) = \lim_{\epsilon \to 0+ } \int_0^\infty e^{ i ( z/x - x (\xi - i \epsilon) ) }
dx & = \lim_{\epsilon \to 0 +} 
-e^{i\varphi_+}\int_0^\infty e^{ i e^{ i \varphi_+}  ( |z|/x 
+ x ( \xi - i \epsilon ) )} dx \\
& = 
-e^{i\varphi_+}|z|^{\frac12}\xi^{-\frac12} 
\int_0^{\infty} e^{i \xi^{\frac12}|z|^{\frac 12} e^{ i \varphi_+} ( 1/t + t ) } 
dt . \end{split}\]
 Since $ 0 < \arg z < \pi $, 
$ \pi/2 < \varphi_+ <  \pi$, and $ \cos (\varphi_+ +\pi/2)= - \cos( \arg ( z )/2 ) < 0 $. 
We can now apply the method of steepest descent to obtain
\begin{gather*} \widehat w ( \xi ) = c_+ \xi^{-\frac34} |z|^{ \frac 14} e^{i e^{ i \varphi_+ } 2 |z|^{\frac12} \xi^{\frac12} }(1+O(|\xi|^{-\frac12})) , 
  \ \  \ \xi \geq 1.
\end{gather*} 
The cases of $ -\xi \geq 1$ is handled similarly except that we deform to $x=e^{i\varphi_-}$ with $\varphi_-={\arg(z)}/{2}$ to obtain
\[ \widehat w ( \xi ) = c_- (-\xi)^{-\frac34} |z|^{ \frac 14} e^{i e^{ i \varphi_- } 2 |z|^{\frac12} (-\xi)^{\frac12} }(1+O(|\xi|^{-\frac12})) %= 
%\mathcal O ( 1) e^{ - 2 \sin ( \arg(z)/2 ) |z|^{ \frac12} \xi^{\frac12} }
, \ \ \  \xi \leq -1.
\] 
% Hence, for $ |\xi| \gg 1 $, 
%\begin{equation}  
%\label{eq:widew} \begin{gathered}
%| \widehat w ( \xi ) | \sim |\xi|^{-\frac34}   e^{ - 2 | \sigma_\pm ( z ) \xi |^{\frac12} } , \\
%\sigma_+(z)=\cos^2(\arg(z)/2), \ \\sigma_-(z)=\sin^2(\arg(z))/2
%\sigma_ ( z ) := \min(\sin^2 ( \arg(z)/2 ),\cos^2(\arg(z)/2)) |z| =\frac{|z|-|\Re z|}{2}.
%\end{gathered} \end{equation}

We now choose  $\chi \in \CIc ( \RR) $ with $ \chi \in \gamma^2 ( \RR ) $
(see \eqref{eq:gev}), 
$ \chi (x )  \equiv 1 $ for $ |x| \leq 1 $, and consider $ u ( x ) 
:= \chi ( x ) w ( x ) $. 
From \eqref{eq:ges}  we see that for any $ \gamma > 0 $
$ |\widehat \chi ( \xi )| \leq C_\gamma e^{ - \gamma |\xi|^\frac12} $.
Hence, for  $ \xi \gg 1 $,  and $ \gamma \gg |z| $, and with 
$  \sigma_+(z) :=\cos^2(\arg(z)/2) |z| $, $  \sigma_-(z): =\sin^2(\arg(z)/2) |z|$, 
we have 
\[ \begin{split} 
&\Big|c_+^{-1}|z|^{-\frac14}\xi^{\frac34}e^{-i e^{ i \varphi_+ } 2 |z|^{\frac12} \xi^{\frac12} } \widehat u ( \xi )- \chi(0)\Big|  \\
& =  \frac{1}{2\pi}\int (c_+^{-1}|z|^{-\frac14}\xi^{\frac34}e^{-i e^{ i \varphi_+ } 2 |z|^{\frac12} \xi^{\frac12} } \widehat{w}(\xi-\zeta)-1)\widehat{\chi}(\zeta) d\zeta\\
&\leq C_\gamma\int_{-\xi/2}^{\xi/2} \frac{|\xi|^{\frac34}-|\xi-\zeta|^{\frac34}}{|\xi-\zeta|^{\frac34}}e^{-2|\sigma_+(z)(\xi-\zeta)|^{\frac12}+2|\sigma_+(z)\xi|^{\frac12}-\gamma |\zeta|^{\frac12}}d\zeta+O(e^{-c\gamma|\xi|^{\frac12}})\\
&\leq C_\gamma|\xi|\int_{-\frac12}^{\frac12} e^{- 2 |\xi|^{\frac12}(|\sigma_+(z)|^{\frac12}(\gamma |r|^{\frac12} + (1-r)^{\frac12}-1 ))}|r|dr+O(e^{-c\gamma|\xi|^{\frac12}})= \mathcal O ( \xi^{-\frac12} ). \\
\end{split} \]
Similarly, for $-\xi\gg 1$, and $\gamma \gg \sigma_-(z)$, 
$$
\Big|c_-^{-1}|z|^{-\frac14}|\xi|^{\frac34}e^{-i e^{ -i \varphi_- } 2 |z|^{\frac12} |\xi|^{\frac12} }\widehat u ( \xi )- \chi(0)\Big|= \mathcal O ( |\xi|^{-\frac12} ).
$$
In particular, with we have 
$
\widehat{u}(\xi)=c_{\pm}|z|^{\frac14}|\xi|^{-\frac34}e^{-2|\sigma_{\pm}(z)\xi|^{\frac12}}(\chi(0)+o(1))$, as $ \pm\xi\to \infty$.
Since $ \min( \sigma+ ( z), \sigma_- ( z )  ) = 
\sigma ( x ) := ( |z| - | \Re z |)/2 $, this and \eqref{eq:ges} give
 \eqref{eq:gez}.
\end{proof}

\end{document}